\documentclass[12pt,a4paper]{amsart} 
\usepackage{amsfonts,amsmath,amscd,amssymb,amstext,amsthm,a4wide,epic,eepic,color,url}
\usepackage[latin9]{inputenc}
\usepackage{esint}
\usepackage[english]{babel}
\usepackage[pdftex]{hyperref}

\theoremstyle{plain}
 \newtheorem{theorem}[subsection]{Theorem}
 \newtheorem{lem}[subsection]{Lemma}
 
 \newtheorem{prop}[subsection]{Proposition}
 \newtheorem{rem}[subsection]{Remark}

 \newtheorem{assumption}[subsection]{Assumption}

 \theoremstyle{definition}
 \newtheorem{defi}[subsection]{Definition}
 
 \newtheorem{notation}[subsection]{Notation}

\newcommand{\rep}{\operatorname {rep}}
\newcommand{\CentAG}{{\Cent_A(G)}}
\newcommand{\CentAoG}{{\Cent_A(\o G)}}

\newcommand{\w}{\widetilde}
\newcommand{\wh}{\widehat}
\newcommand{\Gu}{{\w G}}
\newcommand{\bGu}{{\w \bG}}
\newcommand{\Nu}{{\w N}}
\newcommand{\Bu}{{\w B}}
\newcommand{\la}{\ensuremath{\lambda}}
\newcommand{\Irr}{\mathrm{Irr}}

\newcommand{\Irrl}{\mathrm{Irr}_{\p'}}

\renewcommand{\labelenumi}{(\alph{enumi})}

\renewcommand{\theenumi}{\thesubsection{ }(\alph{enumi})}

\newcommand{\tw}[1]{\,{}^#1\!}
\newcommand{\CC}{\ensuremath{\mathbb{C}}}
\renewcommand{\o}{\overline}
\newcommand{\bB}{{\mathbf B}}
\newcommand{\Cent}{\ensuremath{{\rm{C}}}}

\newcommand{\NNN}{\ensuremath{{\mathrm{N}}}}
\newcommand{\I}{\ensuremath{\mathrm{I}}}

\def\restr#1|#2{\left.#1\right\rceil_{#2}}

\newcommand{\FF}{\ensuremath{\mathbb{F}}}
\newcommand{\SC}{{sc}}
\newcommand{\PSU}{{\mathrm{PSU}}}

\newcommand{\GL}{{\mathrm{GL}}}

\newcommand{\tB}{\mathsf B}
\newcommand{\Cy}{\mathrm C}
\newcommand{\tC}{\mathsf C}
\newcommand{\tD}{\mathsf D}

\newcommand{\tF}{\mathsf F}
\newcommand{\tG}{\mathsf G}
\newcommand{\calD}{\mathcal D}
\newcommand{\D}{\operatorname D}
\newcommand{\calG}{\ensuremath{\mathcal G}}
\newcommand{\al}{{\alpha}}
\newcommand{\si}{{\sigma}}

\makeatletter
\def\Set#1{\Set@h#1@}
\def\Lset#1{\Lset@h#1@}
\def\Set@h#1|#2@{\left\{\left.#1\vphantom{#2}\hskip.1em\,\right|\,\relax #2\right\}}
\def\Lset@h#1@{\left\{#1\right\}}
\makeatother
\newcommand{\Aut}{\mathrm{Aut}}
\newcommand{\Out}{\ensuremath{\mathrm{Out}}}
\def\spann<#1>{\left\langle#1\right\rangle}
\newcommand{\calN}{\ensuremath{\mathcal N }}
\newcommand{\Guchinull} {{\w G_{\chi_0}}}
\newcommand{\cohom}{{\bf (cohom)}}
\newcommand{\Nupsinull} {{\w N_{\psi_0}}}
\newcommand{\Z}{\operatorname Z}
\newcommand{\bG}{\mathbf G}
\newcommand{\bU}{\mathbf U}
\newcommand{\calP}{\mathcal P}
\newcommand{\calQ}{\mathcal Q}
\newcommand{\OmegaNES}{{\Omega_N^\ES}}
\newcommand{\Omegau}{{\w \Omega}}
\newcommand{\ES} {{G}}
\def\cprime{$'$} 
\def\p{p}

\makeindex
\title{Inductive McKay Condition in defining Characteristic}
\author{Britta Sp\"ath}
\address{Institut de Math\'ematiques de Jussieu, Universit\'e Paris VII, 175, rue de Chevaleret, 75013 Paris, France}
\email{spath@math.jussieu.fr}
\thanks{This research has been supported by the Fondation Sciences Math\'ematiques de Paris and the Deutsche Forschungsgemeinschaft, SPP 1388.}

\begin{document}

\begin{abstract}
We reformulate the inductive McKay condition, from Isaacs--Malle--Navarro \cite{IsaMaNa}, and apply the new criterion to simple groups of Lie type, when the prime is the defining characteristic $\p$. Thereby we make use of a recent result of Maslowski \cite{Maslowski}. This proves that these simple groups satisfy the inductive McKay condition for $\p$. In the non simply--laced types and non--classical types this reproves earlier results by Brunat \cite{Brunat,BrunatE6} and Brunat--Himstedt \cite{BrunatHimstedt}. 
\end{abstract}
\maketitle
\section{Introduction}
In the representation theory of finite groups the McKay conjecture asserts that for every finite group $H$ and prime $\p$ the equation \[ |\Irrl(H)|=|\Irrl(N)| \]
holds, where $N$ is the normalizer of a Sylow $\p$--subgroup in $H$ and $\Irrl(X)$ denotes the set of irreducible characters of degree prime to $\p$ for any finite group $X$. In \cite{IsaMaNa}, the authors presented a reduction theorem for it: assuming that every simple group satisfies a certain condition, the now--called {\it inductive McKay condition}, for all primes dividing their order the authors verify the McKay conjecture for every finite group. 

The inductive McKay condition has been checked for sporadic and alternating groups in \cite{ManonLie}. It remains to check the simple groups of Lie type. This paper proves the condition in the following cases.

\begin{theorem}\label{hauptthm}
Let $S$ be a simple group of Lie type of characteristic $p$. Then the inductive McKay condition holds for $S$ and $\p$.
\end{theorem}

An essential part of the inductive McKay condition is that an equivariant bijection exists, which preserves certain cohomological elements. Checking the condition on these elements became quite technical in \cite{Spaeth4}. As an alternative, we introduce here with Proposition \ref{propcohom} another approach. Also, analysing the equivariance of the required bijection is rather difficult, if the bijection depends on various choices, like the one of \cite{Ma06}. Therefore we prove Theorem \ref{hauptthm2}, according to which the existence of the required bijection can be proven without explicit constructions. Likely this statement can be applied for all simple groups of Lie type and the relevant primes.

As a first step we use it for the cases, where the prime is the defining characteristic of the simple group of Lie type. Hereby the ingredients are some basic observations on semisimple characters from Brunat's work and a bijection from \cite{Maslowski}. In the whole we obtain a quite uniform proof for groups, which are not of Suzuki or Ree type, and have no Suzuki or Ree automorphism.

{\bf Acknowledgement.} I am thankful to Johannes Maslowski for providing me with an early version of his thesis and to Olivier Brunat, Michel Enguehard and Gunter Malle for their comments on an earlier version of this manuscript. The author also thanks Marc Cabanes for various discussions and his support.

\section{Criterion for the inductive McKay condition}
In \S 10 of \cite{IsaMaNa} the authors introduced the inductive McKay condition for a simple group and a prime. The verification of the condition on cohomological elements appears especially technical and sophisticated. The aim of this section is to give an alternative approach to the verification of these conditions, which is better adapted to simple groups of Lie type. For this we use the following notation. 

\begin{notation}
Let $Z\lhd G$ be a pair of finite groups. For $\nu \in \Irr(Z)$ the irreducible constituents of the induced character $\nu^G$ form the set $\Irr(G\mid \nu)$, and analogously for $\chi\in \Irr(G)$ the constituents of the restricted character $\restr\chi|Z$ are denoted by $\Irr(Z\mid \chi)$. 
An automorphism $\si$ of $G$ acts by composition on $\Irr(G)$. Let $O$ be a group acting on $G$. Then by $O_U$ we denote the stabiliser in $O$ of a subgroup $U\leq G$, analogously the inertia group $\I_O(\chi)=O_\chi$ ($\chi\in \Irr(G)$) is the stabilizer of $\chi$ in $O$. The stabilizer of the pair $(U,\chi)$ ($\chi\in \Irr(U)\cup \Irr(G)$) is written as $O_{U,\chi}$ and $\Irr_O(G)$ is the set of characters fixed by $O$. 

Recall that for a given prime $\p$ we denote by $\Irrl(G)$ the set of irreducible characters whose degree is not divisible by $\p$. 

Let $A$ be a group such that $G\lhd A$. Let $\chi\in \Irr(G)$ and $\epsilon \in \Irr(\CentAG)$ such that $\epsilon(1)=1$ and $\Irr(\Z(G)\mid \chi)= \Irr(\Z(G)\mid \epsilon)$. Then there exists a unique extension, denoted by  $\chi\cdot \epsilon$, of $\chi$ to $G\CentAG$ with $\Irr(\CentAG\mid \chi\cdot\epsilon)=\Lset{\epsilon}$. 

In case of $A_\chi=A$ there exists, by \cite[11.2]{Isa} a projective representation $\calP:A\rightarrow \GL_{\chi(1)}(\CC)$ of $A$ such that $\restr\calP|G$ is a linear representation affording $\chi$ and its factor set $\al:A\times A\rightarrow\CC$ defined by $\calP(x) \calP(y) =\al(x,y)\calP(xy)$ ($x,y\in A$) is constant on $G$--cosets. One associates to $\al$ a cohomological element in the Schur multiplier of $A/G$, which will be denoted by $[\chi]$, as it is uniquely defined by $\al$, see Chapter 11 of \cite{Isa}. 
By \cite[11.15]{Isa} the projective representation $\calP$ can be chosen such that $\al$ and hence all $\al(x,y)$ ($x,y \in A$) have finite order. In the following we call such a factor set {\it finite}.
\end{notation}

The inductive McKay condition on a simple group $S$ will be checked using the following.
\begin{lem}{{\cite[Remark 5.2]{Spaeth4}}}\label{lem21}
Let $\p$ be a prime and $S$ a simple non--abelian group with $\p\mid |S|$. Let $\ES$ be the maximal perfect central extension of $S$, and $Q$ a Sylow $\p$--subgroup of $\ES$. By $\Aut(G)$ we denote the group of automorphisms of $G$. Assume the following 
\renewcommand{\labelenumi}{(\roman{enumi})}\renewcommand{\theenumi}{\thesubsection{ }(\roman{enumi})}
\begin{enumerate}
	\item \label{lem21_prop_N}
	 there exists an $\Aut(G)_Q$-stable subgroup $N<\ES$ which contains the normalizer of $Q$ in $\ES$, i.e. $\NNN_\ES(Q)\leq N\not= \ES$,
	\item \label{lem21b}
	there exists an $\Aut(G)_Q$--equivariant bijection 
	\[\OmegaNES:\Irrl(\ES)\longrightarrow \Irrl(N),\] with $\OmegaNES(\Irrl(\ES\mid \nu))=\Irrl(N\mid \nu)$ for any character $\nu\in \Irr(\Z(\ES))$, and
	\item \label{lem21cohom}\cohom{} holds for every $(\chi,\OmegaNES(\chi))$ ($\chi\in \Irrl(\ES)$).
\end{enumerate}
\renewcommand{\labelenumi}{(\alph{enumi})}\renewcommand{\theenumi}{\thesubsection{ }(\alph{enumi})}
Then $S$ satisfies the inductive McKay condition for $\p$. 
\end{lem}

Herein one uses the following definition.

\begin{defi}[\cohom]
Let $N< G$ such that \ref{lem21_prop_N} holds. A pair of characters $(\chi,\psi)$ ($\chi\in \Irrl(G)$, $\psi\in \Irrl(N)$) is said to satisfy {\bf \cohom{}} if 
\renewcommand{\labelenumi}{(\roman{enumi})}\renewcommand{\theenumi}{\thesubsection{ }(\roman{enumi})}
\begin{enumerate}
\item $\Irr(\Z(G)\mid\chi)= \Irr(\Z(G)\mid\psi) $,
\item $\Aut(G)_{N,\chi}=\Aut(G)_{N,\psi}$,
\item for $\o G=\ES/\ker(\restr\chi|{\Z(G)})$ and $\o N= N /\ker(\restr\chi|{\Z(G)})$ there exists a group $A$ with $\o G\lhd A$, $A_\chi=A$, abelian $C=\Cent_A(\o G)$ and $A/(\o G C)= \Out(\o G)_\chi$, and further for some $\epsilon\in \Irr_A(C)$ with $\Irr(\Z(\o G)\mid \epsilon)=\Irr(\Z(\o G)\mid \chi)$ the associated elements $[\chi\cdot \epsilon ]$ and $[\psi\cdot \epsilon]$ in the Schur multipliers of $A/{\o G C}$ and $\NNN_A(\o N)/{\o NC}$ respectively, coincide under the natural isomorphism $A/{\o G C}\cong\NNN_A(\o N)/{\o N C}$. (By a Frattini argument $ \NNN_A(Q)\leq N$ implies $A/{\o G C}\cong\NNN_A(\o N)/{\o N C}$, see 10.1 of \cite{IsaMaNa}.)
\end{enumerate}
\renewcommand{\labelenumi}{(\roman{enumi})}\renewcommand{\theenumi}{\thesubsection{ }(\roman{enumi})}
\end{defi}

Particularly in the case where the outer automorphism group $\Out(S)$ is non--cyclic, the definition of the group $A$ for a pair of characters $(\chi,\psi)$ can become difficult, see for example \cite{Spaeth4}. In order to reduce the associated technical problems, we prove in Lemma \ref{lem_cohom_conj} that it suffices to check \cohom{} for certain pairs. Thereby we use projective representations of $\Aut(S)$.
\begin{lem}
There exists an isomorphism $\Aut(\ES) \rightarrow \Aut(S)$, we may embed $\Aut(\ES/K)$ in $\Aut(S)$  for $K\leq \Z(\ES)$ and write $\Aut(S)_\chi:= \Aut(\ES/\ker(\restr \chi|{\Z(\ES)}))_\chi$ for $\chi\in \Irr(\ES)$.
\end{lem}
\begin{proof}
By Exercise 11.6 of \cite{Aschbacher} one has $\Aut(S)\cong\Aut(\ES)$. Further for $K\leq \Z(\ES)$ the natural map $\Aut(\ES/K) \rightarrow \Aut(S)$ is injective because of $\ES=[\ES,\ES]$.
\end{proof}
From now on assume that $N< G$ is a group as in Lemma \ref{lem21_prop_N}. 

\begin{subsection}{\texorpdfstring{Projective representations associated to $(\chi,\psi)$}{Associated projective representations}}\label{subsec_25}
Let $(\chi,\psi)$ ($\chi \in \Irrl(G)$, $\psi\in \Irrl(N)$) be a pair satisfying \cohom{} with the group $A$ and $\rep: \Aut(S)_\chi= A/\Cent_A(\o G)\rightarrow A$ a section of the corresponding natural epimorphism $A \rightarrow A/\Cent_A(\o G)$. Then $\rep(x)\CentAoG = x$. Let $\calP$ be a projective representation of $A$ associated to $\chi\cdot \epsilon$ with factor set $\gamma$, like in \cite[(11.2)]{Isa}. By definition the elements $\calP(a)$ ($a \in \Cent_A(\o G)$) are scalar matrices and we define a projective representation $\w \calP$ of $\Aut(S)_\chi$ by $ \w \calP(x) = \calP(\rep(x))$ for $x\in \Aut(S)_\chi$. In such a situation we say $\w \calP$ is {\it obtained from $\calP$ using $\rep$ and associated to $\chi\cdot \epsilon$}.

As \cohom{} is satisfied we can choose a projective representation $\calP'$ of $A_N=\NNN_A(N)$ associated to $\psi\cdot \epsilon$ with factor set $\gamma$. Let $\w \calP'$ be the projective representation of $\Aut(S)_{N,\psi}$ obtained from $\calP'$ using $\rep$.

\end{subsection}

\begin{lem}\label{lem_prop_cohom}\label{lemnewformcohom}
Let $\w\al$ and $\w \al'$ be the factor set of $\w\calP$ and $\w\calP'$, respectively.
Then \[\w \al'=\restr \w \al|{\Aut(S)_{N,\psi}\times \Aut(S)_{N,\psi}}.\] 
\end{lem}
\begin{proof}
For $x,y\in \Aut(S)_{N,\psi}$ the factor set $\w \al$ of $\w \calP$ satisfies
\[ \w \calP(x)\w \calP(y)= \calP(\rep(x))\calP(\rep(y))=\gamma(\rep(x), \rep(y)) \calP(\rep(x)\rep(y)). \]
and for $z= \rep(x)\rep(y)\rep(xy)^{-1}\in \Cent_A(G)$.
\[ \w \calP(x)\w \calP(y)=\gamma(\rep(x), \rep(y)) \calP(z\rep(xy) )= \gamma(\rep(x), \rep(y)) \epsilon(z)\w\calP(xy ).\]
This shows $\w \al(x,y)= \gamma(\rep(x), \rep(y)) \epsilon(z)$. 
Similarly one can determine the value of $\w \al'$ and obtains the analogous formula. This proves the statement.
\end{proof}

The property of factor sets described in the above lemma implies \cohom{}. 

\begin{lem}\label{lem15}
Let $\rep: S\rightarrow G$ be a section as before, $\chi\in \Irrl(G)$ and $\psi\in \Irrl(N)$ such that $\Aut(S)_{N,\chi}=\Aut(S)_{N,\psi}$ and $\Irr(\Z(G)\mid \chi)=\Irr(\Z(G)\mid \psi)$.
Assume there is a projective representation $\w \calP$  of $\Aut(S)_\chi$ with finite factor set $\w\al$, such that $\calP_0:=\restr\w\calP|S$ is a projective representation associated to $\chi$ using $\rep$. Assume there is a projective representation  $\w \calP'$ of $\Aut(S)_{N,\psi}$ with factor set $\w \al'$, such that  $\calP'_0:= \restr \w \calP'|{N/\Z(G)}$ is a projective representation associated to $\psi$ using $\rep$. If $\w \al'=\restr{\w \al}|{\Aut(S)_{N,\psi} \times \Aut(S)_{N,\psi}}$ then the pair $(\chi,\psi)$ satisfies \cohom{}. 
\end{lem}
\begin{proof}
Let $\Cy_e\leq \CC^*$ be a finite cyclic group such that $\w \al(x,x')\in \Cy_e$ for all $x,x'\in \Aut(S)_\chi$. The factor set $\w \al$ defines a central extension $E$ of $\Aut(S)_\chi$: the elements of $E$ are $(a,c)$ ($a\in \Aut(S)_\chi$, $c\in \Cy_e$) and the multiplication is given by $(a,c)(a',c')= (aa', \w \al (a,a') cc')$ for $a,a'\in \Aut(S)_\chi$ and $c,c'\in \Cy_e$. The projective representation $\w \calP$ can be lifted to a linear representation $\w \calD$ of $E$, which is defined by $\w \calD(a,c)=c \w \calP(a)$ ($a\in A$, $c\in \Cy_e$).

For $\nu\in \Irr(\Z( G)\mid \chi)$  and $Z:=\ker(\nu)$ the map \[ f:\o G:= G/Z\rightarrow E \text{ with } \rep(s)z\mapsto (s,\nu(z)) \] defines an injective morphism. The linear representation $\restr\w \calD|{f(\o G) \Cy_e}$ affords $\chi\cdot\epsilon$ with $\epsilon(c)=c$ and therefore the associated cohomological element $[\chi\cdot \epsilon]$ is trivial. The group $E$ satisfies the desired group theoretic properties of $A$ in \cohom{}.

Now we consider the local situation: the isomorphism $\NNN_E\left(f(N/Z)\right)\Cy_e/\Cy_e \cong \Aut(S)_{N,\psi}$ holds. Because of $\w \al'=\restr \w \al|{\Aut(S)_{N,\psi} \times \Aut(S)_{N,\psi}}$ we can define a linear representation of $\NNN_E(f(N/Z))$ by 
\[ \w \calD'(a,c)= c \w \calP'(a). \]
By definition the afforded character on $f(N)\Cy_e$ is $\psi\cdot \epsilon$ and extends to $E$, hence $[\psi\cdot\epsilon]=1$. This implies that $(\chi,\psi)$ satisfies \cohom{}. 
\end{proof}

It is clear that an automorphism $\si\in \Aut(S)= \Aut(\ES)$ acts on projective representations of $\si$--stable subgroups of $\Aut(S)$. 
\begin{lem}\label{lem_cohom_conj}
Let $\si\in \NNN_{\Aut(S)}(N/\Z(G))$, $\chi\in \Irrl(\ES)$ and $\psi\in \Irrl(N)$, such that $(\chi, \psi)$ satisfies \cohom{}. Then \cohom{} holds for $(\chi^\si,\psi^\si)$.
\end{lem}
\begin{proof}
Let $A$, $\rep: \Aut(S)_\chi\rightarrow A$,  $\calP$, $\w \al$, $\calP'$ and $\w \al'$ be associated to $(\chi, \psi)$ as in \ref{subsec_25} and Lemma \ref{lem_prop_cohom}.
For $\chi^\si$ we have the following: there is a projective representation $\calQ$ of $(\Aut(S)_\chi)^\si=\Aut(S)_{\chi^\si}$ with 
$ \calQ(a)=\calP(a^{\si^{-1}})$, where $\restr \calQ|S $ is a projective representation of $S$.
Using $\rep': S\rightarrow G$ with $\rep'(s^\si)=\rep(s)^\si$ lifts this to a linear representation of $G$ affording $\chi^\si$. The factor set $\w \beta$ of $\calQ$ is finite and $ \w \beta(a^\si,{a'}^\si)= \w \al(a,a')$.
Further the projective representation $\calQ'$ of $(\Aut(S)_{N,\psi})^\si=\Aut(S)_{N,\psi^\si}$ with 
$ \calQ'(a)=\calP'(a^{\si^{-1}})$
extends a projective representation of $N/\Z(G)$, which can be obtained from a linear representation affording $\psi^\si$ using $\rep'$. The corresponding factor set is $\w \beta'$ with $ \w \beta'(a^\si,{a'}^\si)= \w \al'(a,a')$ for $a,a'\in \Aut(S)_{N,\psi}$.
It follows that the equation $\restr \w \beta|{\Aut(S)_{N,\psi^\si} \times \Aut(S)_{N,\psi^\si}}=\w \beta'$ holds and by Lemma \ref{lem15} the pair $(\chi^\si,\psi^\si)$ satisfies \cohom.
\end{proof}

In order to simplify checking \cohom{} we present another approach. As one can see for $S=\PSU_n(q)$ with $\p\mid (q-1)$ in \cite{Spaeth4} it is quite difficult to define a group $A$ with the desired properties, especially such that $\chi$ extends to an $A$--invariant character of $\o G \CentAoG$. By the next statement this requirement is replaced by conditions on projective representations. 

\begin{prop} \label{propcohom}
Let $\chi\in \Irrl(G)$, $\psi\in \Irrl(N)$ with $\Irr(\Z(G)\mid\chi)= \Irr(\Z(G)\mid\psi)$, and $\Aut(G)_{N,\chi}=\Aut(G)_{N,\psi}$. Further let $A$ be a group with $G\lhd A$ and the following properties:
\renewcommand{\labelenumi}{(\roman{enumi})}\renewcommand{\theenumi}{\thesubsection{ }(\roman{enumi})}
\begin{itemize}
\item $A/\Cent_A(G)\cong \Aut(S)_\chi$ via the canonical isomorphism,
\item there exists a projective representation $\calP$ of $A$ extending a linear representation of $G$ affording $\chi$, with finite factor set $\beta$ and $\calP(ca)= \epsilon(c)\calP(a)$ for every $c\in \Cent_A(G)$, where $\epsilon\in \Irr(\Cent_A(G))$ is a linear character with $\Irr(\Z(G)\mid \chi)=\Irr(\Z(G)\mid \epsilon)$, and 
\item there exists a projective representation $\calP'$ of $A':=\NNN_A(N)$ extending a linear representation of $N$ affording $\psi$, with finite factor set $\restr\beta|{A'\times A'}$ and $\calP'(ca)= \epsilon(c)\calP'(a)$ for every $c\in \Cent_A(G)$ and $a\in A'$.
\end{itemize} 
Then $(\chi,\psi)$ satisfies \cohom.
\end{prop}
\begin{proof}
Using a section $\rep:\Aut(S)_\chi \rightarrow A$ as before one obtains a projective representation $\calQ$ of $\Aut(S)_\chi$ via
\[ \calQ(\o a) = \calP(\rep(\o a)) \text{ for } \o a \in \Aut(S)_\chi,\]
as every $\calP(c)$ ($c\in \Cent_A(G)$) is a scalar matrix. 
The associated factor set $\al:\Aut(S)_\chi\times \Aut(S)_\chi\rightarrow \CC^*$  satisfies 
\[ \al(\overline a, \overline a')= \beta(a,a') \epsilon( aa' \rep(\overline a \overline a')^{-1}),\]
for every $\overline a, \overline a' \in \Aut(S)_\chi$, $a:=\rep(\overline a)$ and $a':=\rep(\overline a')$. 

Analogously the projective representation $\calQ'$ of $\Aut(S)_{N,\psi}$ with \[ \calQ'( \o a) = \calP'(\rep(\o a)) \text { for } \overline a\in \Aut(S)_{N\psi}.\]
has a factor system $\al'$ with
\[ \al'(\overline a, \overline a')= \beta(a,a') \epsilon( aa' \rep(\overline a \overline a')^{-1}),\]
for $\overline a, \overline a' \in \Aut(S)_{N,\psi}$, $a:=\rep(\overline a)$ and $a':=\rep(\overline a')$. This implies $\al'=\restr\al|{\Aut(S)_{N,\psi}\times \Aut(S)_{N,\psi}}$. By Lemma \ref{lem15} this proves the statement.
\end{proof}

Finally we state a quite technical theorem, which we apply in the proof of our main result. One reconstructs the bijection needed in Lemma \ref{lem21} by a second one, whose construction is easier in the situation of Theorem \ref{hauptthm}. Here we use the term of maximal extensibility: for two finite groups $X_1\lhd X_2$ {\it maximal extensibility} holds, if every character $\chi\in \Irr(X_1)$ can be extended as irreducible character to its inertia group in $X_2$.

\renewcommand{\labelenumi}{(\roman{enumi})}\renewcommand{\theenumi}{\thesubsection{ }(\roman{enumi})}
\begin{theorem}\label{hauptthm2}
Let $G$ be the maximal central perfect extension of a simple non--abelian group $S$ and $Q$ a Sylow $\p$--subgroup for some prime  $\p$ with $\p\mid |S|$.

Assume the following: 
\begin{enumerate}
\item there exists a group $\Gu$ with $G\lhd \Gu$ and a finite group $D$ such that $\Gu\rtimes D$ is defined, $\Gu/G$ is abelian, $\Cent_{\Gu\rtimes D}(G)=\Z(\Gu)$, $G\lhd \Gu\rtimes D$ and $\Gu\rtimes D$ induces on $G$ all automorphisms of $G$,
\item there exists a $(\Gu\rtimes D)_Q$-stable group $N< G$ with $\NNN_G(Q)\leq N$ and $N\neq G$,
\item \label{hauptprop_maxext}maximal extensibility holds for $G\lhd \Gu$ and $N\lhd \Nu:=\NNN_\Gu(N)$,
\item there exists a $(\Gu\rtimes D)_Q$--equivariant bijection $ \Omegau: \calG\rightarrow \calN$ for 
$ \calG=\break\Set{\chi \in \Irr(\Gu)|\Irr(G\mid \chi) \subset \Irrl(G)}$ and $ \calN=\Set{\psi \in \Irr(\Nu)|\Irr(N\mid \psi) \subset \Irrl(N)}$ with 
\begin{itemize}
\item $ \Omegau(\calG\cap\Irr(\Gu\mid \nu))= \calN\cap\Irr(\Nu\mid \nu)$ for every $\nu \in \Irr(\Z(\Gu))$,
and
\item \label{Omega_u_epsilon_equiv} 
$\Omegau(\chi\la)= \Omegau(\chi)\restr \la|{\Nu}$ for every $\la \in \Irr(\Gu/G)$.
\end{itemize}
\item \label{Star-cond-global}
for every $\chi\in \calG$ there exists some $\chi_0\in \Irrl(G\mid \chi)$ such that $(\Gu\rtimes D)_{\chi_0}= \Guchinull\rtimes D_{\chi_0}$ and $\chi_0$ extends to $(G \rtimes D)_{\chi_0}$, and
\item \label{Star-cond-local}
for every $\psi\in \calN$ there exists some $\psi_0\in \Irr(N\mid \psi)$ such that the group $O:=(\Gu\rtimes D)_{N,\psi_0}G$ satisfies $O= (\Gu\cap O) \rtimes (D\cap O)$ and $\psi_0$ extends to $(G\rtimes D)_{N,\psi_0}$,
\end{enumerate}
Then the inductive McKay condition holds for $S$ and $\p$.
\end{theorem} 

Before constructing a bijection $\OmegaNES: \Irrl(G)\rightarrow \Irrl(N)$ with the properties required in Lemma \ref{lem21b}, we check \cohom{} for some pairs. 
\begin{lem}
Assume the situation of Theorem \ref{hauptthm2}. Let $\chi\in \calG$, $\psi:=\Omegau(\chi)$, $\chi_0\in \Irrl(G\mid \chi)$ from \ref{Star-cond-global} and $\psi_0\in \Irrl(G\mid \psi)$ from \ref{Star-cond-local} with $(\Gu\rtimes D)_{N,\chi_0}= (\Gu\rtimes D)_{N,\psi_0}$. Then the pair $(\chi_0,\psi_0)$ satisfies \cohom.
\end{lem}
\begin{proof}
Let $\w\chi_0$ be the extension of $\chi_0$ to $\Guchinull$ with $\w\chi_0^\Gu=\chi$, and $\w\psi_0\in \Irr(\Nupsinull\mid \psi_0)$ with $\w\psi_0^\Gu=\psi$. They exist by \ref{hauptprop_maxext}, as $\Gu/G$ is abelian. Because of the properties of $\Omegau$ we have $\Irr(\Z(\Gu)\mid \chi)=\Irr(\Z(\Gu)\mid \psi)$. Together with $\Cent_{\Gu\rtimes D}(G)=\Z(\Gu)$ this implies
\[ \Irr(\Cent_{\Gu\rtimes D}(G)\mid \w \chi_0)= \Irr(\Z(\Gu)\mid \chi)= \Irr(\Z(\Gu)\mid \psi)= \Irr(\Cent_{\Gu\rtimes D}(G)\mid \w \psi_0)= \Lset{\epsilon}\]
for some $\epsilon\in \Irr(\Z(\Gu))$. 

Let $\calD_1$ be a linear representation of $\Guchinull$ affording $\w\chi_0$ and $\calD_2$ a linear representation of $G\rtimes D_{\chi_0}$ with $\restr\calD_1|{G}=\restr\calD_2|{G}$. By 
\[\calP(x_1x_2)=\calD_1(x_1) \calD_2(x_2) \text{ for } x_1\in \Guchinull\text{, } x_2\in (G\rtimes D)_{\chi_0}\]
one defines a projective representation of $(\Gu\rtimes D)_{\chi_0}$, which is well--defined because of $\restr\calD_1|{G}=\restr\calD_2|{G}$. Further $\restr\calP|G$ is a linear representation affording $\chi_0$ 
and satisfies
\[ \calP(cx)=\epsilon(c) \calP(x) \text{ for }c\in \Cent_{\Gu\rtimes D}(G), \,\, x \in (\Gu\rtimes D)_{\chi_0}.\]
By the proof of Lemma 5.5 in \cite{Spaeth4} the factor set $\al$ of $\calP$ satisfies 
\[\al(x_1x_2, x_1'x_2') = \mu (x_1') \text{ for } x_1, x_1'\in \Guchinull\text{, } x_2, x_2'\in (G\rtimes D)_{\chi_0}, \]
where $\mu \in \Irr(\Guchinull/G)$ satisfies  $\w \chi_0= \mu \w\chi_0^{x_2}$. (By \cite[6.17]{Isa} this defines $\mu$ uniquely.) 

Analogously in the local situation, there exists a projective representation $\calP'$ of $(\Gu\rtimes D)_{N,\psi}$ such that $\restr \calP'|N$ is a linear representation affording $\psi$, it satisfies \[ \calP(cx)=\epsilon(c) \calP(x) \text{ for }c\in \Cent_{\Gu\rtimes D}(G), \,\, x \in (\Gu\rtimes D)_{N,\psi_0}\]
and its factor set $\al'$ is determined by 
\[\al'(x_1x_2, x_1'x_2') = \mu' (x_1') \text{ for } x_1, x_1'\in \Nupsinull\text{, } x_2, x_2'\in (G\rtimes D)_{N,\psi_0}, \]
where $\mu' \in \Irr(\Nupsinull/N)$ is defined uniquely by $\w \psi_0= \mu' \w\psi_0^{x_2}$. 

For $x_2\in (G\rtimes D)_{N,\chi_0}=(G\rtimes D)_{N,\psi_0}$ the character $\mu$ has an extension $\w \mu \in \Irr(\Gu/G)$ with $\chi=\w \mu \chi^{x_2}$. Using the equivariance of $\Omegau$ we obtain $\psi= \restr\w \mu|{\Nu} \psi^{x_2}$ and $\restr \w \mu|{\Nupsinull}= \mu'$. Hence $\al'$ is the restriction of $\al$. By Proposition \ref{propcohom} this implies that $(\chi_0, \psi_0)$ satisfies \cohom.
\end{proof}
We use this result in the proof of the Theorem. 
\renewcommand{\proofname}{Proof of Theorem \ref{hauptthm2}}
\begin{proof}
The desired bijection will be explicitly constructed, but some ad hoc  choices are necessary. We begin by constructing a $(\Gu\rtimes D)_Q$--equivariant bijection $\OmegaNES:\Irrl(G)\rightarrow \Irrl(N)$.

On $\calG$ the group $(\Gu\rtimes D)_Q$ acts by conjugation and the group of linear characters $\Irr(\Gu/G) $ by multiplication. Let $\calG_T$ be a transversal in $\calG$ with respect to these combined actions. For every $\chi \in \calG_T$ we fix a character $\chi_0 \in \Irrl( G \mid \chi)$ with the properties from \ref{Star-cond-global} and let $\calG_0\subseteq \Irrl(G)$ be the set formed by them. By definition this is a $(\Gu\rtimes D)$--transversal in $\Irrl(G)$.

Because of the equivariance properties of $\Omegau$ the set $\calN_T:=\Omegau(\calG_T)$ is a transversal in $\calN$ with respect to the action of $\Irr(\Gu/G)$ and $(\Gu\rtimes D)_Q$. Like before one can associate to every $\psi\in \calN_T$ a character $\psi_0\in \Irrl(N\mid \psi)$ with the properties mentioned in \ref{Star-cond-local}. Let these character form the set $\calN_0$, which is a $(\Gu\rtimes D)_N=N(\Gu\rtimes D)_Q$--transversal in $\Irrl(N)$. 

For $\chi\in \calG_T$ and $\chi_0\in \calG_0 \cap \Irr(G\mid \chi)$ we define $\OmegaNES(\chi_0):=\psi_0$ where $\psi_0$ is the unique element in $\calN_0\cap\Irr(N\mid \Omegau(\chi))$. If $ (\Gu\rtimes D)_{N,\chi_0}=(\Gu\rtimes D)_{N,\psi_0}$ for every $\chi_0 \in \calG_0$ and $\psi_0=\OmegaNES(\chi_0)$, 
one can define a $(\Gu\rtimes D)_Q$--equivariant bijection $\OmegaNES:\Irrl(G)\rightarrow \Irrl(N)$ 
 by \[\OmegaNES(\chi_0^x):= \OmegaNES(\chi_0)^x\text{ for every } x \in (\Gu\rtimes D)_Q \text{ and } \chi_0 \in \calG_0.\]

We want to show that $(\Gu\rtimes D)_{N,\chi_0}=(\Gu\rtimes D)_{N,\psi_0}$. Because of \ref{Star-cond-global} we have
\[ (\Gu\rtimes D)_{\chi_0}= \Guchinull \rtimes D_{\chi_0}. \]
The group $D_{\chi_0}$ satisfies $D_{\chi_0}= \Set{ d\in D|{ \chi^ d= \chi \la \text{ for some } \la \in \Irr(\Gu/G)}}$.
For $\Guchinull$ we obtain the formula 
\[ \Guchinull = \bigcap_{\Set{\la \in \Irr(\Gu/G) |\chi\la =\chi}} \ker(\la), \]
since for $\la \in \Irr(\Gu/G)$ the equation $\la \chi=(\restr \la|{\Guchinull} \w \chi_0)^\Gu$ holds, where $\w \chi_0\in \Irr(\Guchinull)$ is the extension of $\chi_0$ with $\w \chi_0^{\Gu}=\chi$.

Analogously we have for $O= (\Gu\rtimes D)_{N,\psi_0}G$ the equation 
\[ O= (\Gu\cap O) \rtimes (D\cap O)\]
by \ref{Star-cond-local}. The group $(G\rtimes D)\cap O$ is determined by
\[(G\rtimes D) _{N,\psi_0}= \Set{ d\in (G\rtimes D)_N|{ \psi^ d= \psi \la \text{ for some } \la \in \Irr(\Nu/N)}}. \]
For $\Gu\cap O= G\Nupsinull$ we use the formula
\[ \Nupsinull= \bigcap_{\Set{\la \in \Irr(\Nu/N) |\psi\la =\psi}} \ker(\la), \]

According to \ref{Omega_u_epsilon_equiv} the sets $\Set{\la \in \Irr(\Nu/N) |\psi\la =\psi}$ and $\Set{\la \in \Irr(\Gu/G) |\chi\la =\chi}$ coincide under $\Gu/G\cong\Nu/N$. This implies $\Guchinull=G\Nupsinull$. As $\Omegau$ is equivariant and the equation $G\NNN_{G\rtimes D}(N)=G\rtimes D$ holds the following calculation is possible
\begin{eqnarray*}
G (G\rtimes D) _{N,\psi_0} &=& G \Set{ d\in (G\rtimes D)_N|{ \psi^ d= \psi \la \text{ for some } \la \in \Irr(\Nu/N)}} =\\
&=& G \Set{ d\in (G\rtimes D)_N|{ \chi^ d= \chi \la \text{ for some } \la \in \Irr(\Gu/G)}} =\\
&=& G \rtimes D_{\chi_0}.
\end{eqnarray*}
This proves $(\Gu\rtimes D)_{N,\chi_0}=(\Gu\rtimes D)_{N,\psi_0}$ and $\OmegaNES:\Irrl(G)\rightarrow \Irrl(N)$ can be defined as mentioned before. The map $\OmegaNES$ is a $(\Gu\rtimes D)_Q$--equivariant bijection with \[ \OmegaNES(\Irrl(G\mid \nu))= \Irrl(N\mid \nu)\text{ for every } \nu \in \Irr(\Z(G)).\]

To apply Lemma \ref{lem21}, it remains to ensure the assumptions in \ref{lem21cohom}. 
By the previous lemma the pair $(\chi, \OmegaNES(\chi))$ ($\chi\in \calG_0$) satisfies \cohom{} and by Lemma \ref{lem_cohom_conj} this implies \cohom{} for all pairs $(\chi, \OmegaNES(\chi))$ ($\chi\in \Irrl(G)$).
\end{proof}
\renewcommand{\proofname}{Proof}

\section{Results in the defining characteristic case}
In this section we apply Theorem \ref{hauptthm2} in the case where $S$ is a simple group of Lie type and $\p$ its defining characteristic. 

\begin{notation}\label{Notation_Gu}
Let $\bG$ be a simply--connected simple algebraic group defined over a field $\FF_q$ with $q$ elements of characteristic $p$ by the Frobenius endomorphism $F:\bG\rightarrow \bG$. Let $\iota: \bG\rightarrow \bGu$ be a regular embedding and $F:\bGu\rightarrow \bGu$ a Frobenius endomorphism extending $F:\bG\rightarrow \bG$, see  \cite[\S 15]{CabEng}. 
Let $F_0$ and $\gamma$ be automorphisms of $\Gu=\bGu^F$ stabilizing $G:=\bG^F$, such that $F_0$ and $\gamma$ induce a field automorphism defining an $\FF_p$--structure, and a graph automorphism of $\bG^F$, respectively, and both automorphisms are as in Theorem 2.5.1 of \cite{GLS3}. Let $D=\spann<F_0,\gamma>$, if $\bG^F\neq \tD_{4,\SC}(q)$. In the case of $\bG^F=\tD_{4,\SC}(q)$ let $\gamma_i$ for $i=2,3$ be a graph automorphisms of $\bG^F$ inducing an automorphism of order $i$ as in \cite[2.5.1]{GLS3} and $D=\spann<F_0,\gamma_2, \gamma_3>$. 
Assume that $F_0, \gamma$ and $\gamma_i$ induce on $\Gu$ and $G$ automorphisms of the same order. Let $\bB$ be a $D$--stable Borel of $\bG$ and $\bU$ its unipotent radical.
\end{notation}
For simplification we exclude some cases such that the group $[P,P]$ of the Sylow $\p$--subgroup $P$ of $G$ has a uniform structure and the set $\Irrl(G)$ coincides with the set of semisimple characters.
\begin{assumption}\label{ass_Gu}
The group $G=\bG^F$ satisfies 
\[G\notin\Lset{\tB_n(2), \tC_n(2), \tG_2(2), \tB_2(2^i), \tG_2(3^i), \tF _4(2^i) }.\] 
\end{assumption}
In the following we construct a $D$--invariant Gel\cprime fand--Graev character of $G=\bG^F$, which exists by Lemma 3.1 of \cite{BrunatHimstedt} in some cases. We use the notation introduced in \cite{Maslowski} for linear characters of $\bU^F$.
\begin{defi}\label{phiS_def}
Let $(\FF_{q^l},+)$($l\in \Lset{1,2,3}$) be the additive group in $\FF_{q^l}$, and $\phi_l:(\FF_{q^l},+) \rightarrow \CC$ a non--trivial linear character, which is invariant under all field automorphisms of $\FF_{q^i}$. By linear algebra such a character exists.
Let $R^+$ and $R_F$ be the set of positive, respectively simple roots associated to the Borel $\bB$. The Frobenius endomorphism $F$ acts by permutation on $R_F$. The elements of $\bU$ are of the form $\prod_{ \al \in R^+}x_\al(t_\al)$ for some $t_\al$ in the algebraic closure of $\FF_q$ as in Theorem 1.12.1 of \cite{GLS3}.

It is well--known, that for an $F$--transversal $I=\Lset{i_1,\ldots, i_{|I|}}$ in $R_F$ the map defined by 
\[ \bU^F \ni \prod_{\al\in R^+}x_\al(t_\al) \mapsto (t_{i_1},\ldots, t_{i_{|I|}}) \in (\FF_{q^l},+)^{|I|} \text{ for }\prod_{\al\in R^+}x_\al(t_\al) \in \bU^F \]
is a group morphism with kernel $[\bU^F,\bU^F]$, where $l$ is the maximal length of an $F$--orbit in $R_F$. 
For any subset $S\subseteq I\subseteq R_F$ there exists a unique linear character $\phi_S\in \Irr(\bU^F)$ with 
$\phi_S(\prod_{\al\in R^+}x_\al(t_\al))= \prod_{\al\in S}\phi_l(x_\al(t_\al))$.
The character $\phi_I$ is regular, see 14.B in \cite{Cedric}, and $D$--invariant. Hence the associated Gel\cprime fand Graev character $\Gamma= (\phi_{R_F})^G $ of $G=\bG^F$ is also $D$--invariant.
\end{defi}
\renewcommand{\labelenumi}{(\alph{enumi})}\renewcommand{\theenumi}{\thesubsection{ }(\roman{enumi})}

\begin{rem}\label{rem_semisimple_char}
\begin{enumerate}
\item The groups $\Gu=\bGu^F$ and $D$ satisfy the group theoretic assumptions of Theorem \ref{hauptthm2}.
\item Maximal extensibility holds for $G\lhd \Gu$.
\item For every $\chi\in \calG$ there exists some $\chi_0\in \Irr(\Gu\mid \chi)$ with $(\Gu\rtimes D)_{\chi_0}= \Guchinull\rtimes D_{\chi_0}$ and $\chi_0$ extends to $G\rtimes D_{\chi_0}$.
\end{enumerate}
\end{rem}
\begin{proof}
Part (a) is an immediate consequence of Theorem 2.5.1 of \cite{GLS3} as we have excluded the groups that have an Suzuki or Ree  automorphism. Part (b) is Theorem 15.11 of \cite{CabEng}. By the proof of Lemma 5 in \cite{Brunat} the set $\Irrl(G)$ coincides with the set of semisimple characters. The set $\calG$ from Theorem \ref{hauptthm2} is the set of semisimple characters of $\Gu$ and because of $\p\nmid|\Gu/G|$ it coincides with $\Irrl(\Gu)$. By a result of Asai \cite[15.11]{Cedric}, there exists exactly one $\chi_0\in \Irr(G\mid \chi)$ with multiplicity $\pm 1$ in $\D_G(\Gamma)$, where $\D_G$ is the Alvis Curtis duality. By definition $\D_\bG$ commutes with the automorphisms of $D$, hence $\D_G(\Gamma)$ is also $D$--stable. Together with the fact that every $\Gu$--orbit in $\Irrl(G)$ contains exactly one constituent of $\D_\bG(\Gamma)$ this implies \[ (\Gu\rtimes D)_{\chi_0}=\Guchinull\rtimes D_{\chi_0}.\]

We extend $\chi_0$ to  $G\rtimes D_{\chi_0}$ using the methods described in Section 3.2 of \cite{BrunatE6}. By (11.31) of \cite{Isa} it suffices to check that $\chi_0$ extends to $G\rtimes D_{\chi_0,i}$ for $i\in \Lset{2,3}$, where $D_{\chi_0,i}$ is a (non-cyclic) Sylow $i$--subgroup $D_{\chi_0,i}$ of $D_{\chi_0}$. The only non--cyclic Sylow $i$--subgroups of $D$ are of the form $\Cy_i \times \spann<F_0>$, where $\Cy_i$ is generated by a graph automorphism of order $i$. Hence if $D_{\chi_0,i}$ is non--cyclic a graph--automorphism $\si$ of order $i$ lies in $D_{\chi_0,i}$. According to (34) of \cite{BrunatE6} there exists a $D_{\chi_0}$--stable extension $\wh \Gamma$ of $\Gamma$ to $G\rtimes \spann<\si>$. As $\Gamma$ is multiplicity free, the extension $\wh \Gamma$ contains a unique extension $\zeta$ of $\D_\bG(\chi_0)$, which is hence also $D_{\chi_0}$--stable. As $\si$ is a quasi--central automorphism of $\bG$ we can apply Proposition 3.13 of \cite{grnc} and there exists an extension $\wh \chi_0$ of $\chi_0$ to $G\rtimes \spann<\si>$ which can be obtained from $\zeta$ using the isometric involution $\D_{\bG.\si^j}$ from Definition 3.10 of \cite{grnc}. The map $\D_{\bG.\si^j}$ is defined using Harish--Chandra induction and restriction, hence it commutes with the action of $D$, the character $\wh \chi_0$ is $D_{\chi_0,i}$--invariant and extends to $G\rtimes D_{\chi_0,i}$, as the quotient $D_{\chi_0,i}/\spann<\si>$ is cyclic. 
\end{proof}

 Now we consider the local situation: by 2.29 (i) of \cite{CabEng} the group $B:=\bB^F$ is the normalizer of a Sylow $\p$--subgroup of $G$ and serves as $N$, when proving Theorem \ref{hauptthm}. In \cite{Maslowski} the author chooses $\bGu$, which is called $\bG_u$ there, and $F$, as well as $F_0$ and $\gamma$ on $\bGu$ in a certain way, see Sections 2 and 3 of \cite{Maslowski}. For all his results we assume $\bGu$, $F$, $F_0$ and $\gamma$ to be defined like there. Likely the results generalize to arbitrary choices of $\bGu$ and $F$. Further in \cite{Maslowski} the author excludes $G\in \Lset{\tD_n(2), \tw 2 \tD_n(2)}$ but the used assumptions on $\bU^F/[\bU^F,\bU^F]$ and $\p'$--characters are also satisfied in this case. Hence his proof applies there, as well. 

\begin{theorem}[Maslowski]\label{thm_Maslowski}
Let $\bGu$ and $F$ be defined as in \cite{Maslowski}.
\begin{enumerate}
\item Maximal extensibility holds with respect to $B\lhd \Bu$.
\item There exists a $D$--equivariant bijection $\Omegau: \Irrl(\Gu)\rightarrow \Irrl(\Bu)$ for $\Bu:=\NNN_\Gu(B)$ with $\Omegau(\chi\la)=\Omegau(\chi)\restr\la|{\Bu}$ for $\la\in \Irr(\Gu/G)$, and 
$\Irr( \Z(\Gu)\mid \chi)= \Irr( \Z(\Gu)\mid \Omegau(\chi))$. 
\end{enumerate}
\end{theorem}
\begin{proof}
Part (a) is contained in the proof of Proposition 11.3 of \cite{Maslowski}. The $D$--equivariant bijection exists by \cite[Theorem 15.1]{Maslowski}. Theorem 15.3 of \cite{Maslowski} and some straight--forward  calculations prove $\Omegau(\chi\la)=\Omegau(\chi)\restr\la|{\Bu}$ for $\la\in \Irr(\Gu/G)$. The equation 
$\Irr( \Z(\Gu)\mid \chi)= \Irr( \Z(\Gu)\mid \Omegau(\chi))$ is Proposition 15.2 of \cite{Maslowski}. 	
\end{proof}
It remains to check the assumption in Proposition \ref{Star-cond-local}. 
\begin{rem} \label{rem_Star_cond_local}
Let $\bGu$ and $F$ be defined as in \cite{Maslowski}.
\begin{enumerate}
\item {\rm (Maslowski)} For every $\psi\in \Irrl(\Bu)$ there exists some $\psi_0\in \Irrl(B\mid \psi)$ such that $(\Bu\rtimes D)_{\psi_0}= \Bu_{\psi_0}\rtimes D_{\psi_0}$. 
\item The character $\psi_0$ extends to $(B\rtimes D)_{\psi_0}$.
\end{enumerate}
\end{rem}
\begin{proof}
Part (a) is Proposition 11.13 (3) of \cite{Maslowski}. The statement (b) is trivial if $D_{\psi_0}$ is cyclic. Hence we may assume that $F$ is not a standard Frobenius endomorphism. By Proposition 11.13 (3) of \cite{Maslowski} the character $\psi_0$ from (a) can be chosen to lie above a $F_0$--invariant character $\phi_S\in \Irr(\bU^F)$ for some set $S\subseteq R_F$, where $\phi_S$ is defined as in Definition \ref{phiS_def}.

Let $F'\in \spann<F_0>$ such that $\spann<F'>=\spann<F_0>_{\psi_0}$. Then the extension $\w \phi_S$ of $\phi_S$ to $\I_B(\phi_S)$ with $(\w \phi_S)^B=\psi_0$ must satisfy $(\w \phi_S)^{F'}=\w \phi_S$. Let $\widehat \phi_S$ be an extension of $\w \phi_S$ to $\I_B(\phi_S)\rtimes \spann<F'>$. 
For a graph automorphism $\gamma$ we have $(\phi_S)^\gamma=\phi_{\gamma(S)}$ and $\Irrl(B\mid \phi_S)\cap \Irrl(B\mid \phi_{S'})=\emptyset$ for $S\ne S'$ by Proposition 8.4 of \cite{Maslowski}. Hence $(\phi_S)^\si=\phi_{S}$ and $(\w \phi_S)^\si=\w \phi_S$ hold for any $\si\in D_{\psi_0}$. As $\widehat \phi_S$ is a linear character, the equation $[F',\si]=1$ implies that the characters $\widehat \phi_S$ and $(\wh \phi_S)^{B\rtimes \spann<F'>}$ are $\si$--invariant. This proves that $\psi_0$ has an extension to $B\rtimes D_{\psi_0}$.
\end{proof}
\renewcommand{\proofname}{Proof of Theorem \ref{hauptthm}}

Now we have checked all assumptions of Theorem \ref{hauptthm2} and may conclude with the proof of our main theorem.
\begin{proof}
Every simple group of Lie type is the quotient of a Suzuki or Ree group or is isomorphic to $G/\Z(G)$ for some $G$ of Notation \ref{Notation_Gu}. 
Let us first consider the cases, excluded in Assumption \ref{ass_Gu}.
For $G=\tB_n(2) \cong \tC_n(2)$ ($n>2$), $G\in \Lset{\tB_2(2^i),\tF_4(2^i),\tG_2(3^i) }$ ($i>1$) and for Suzuki and Ree groups the statement is known by \cite{Cabanes} and \cite{Brunat}, respectively. One can find references or proofs for the remaining exceptions in  \cite[Remark 2]{Brunat}. By Theorem 1.1 of \cite{ManonLie} every simple group $S$ of Lie type with an exceptional Schur multiplier satisfies the inductive McKay condition. Hence we may assume that $G$ is the maximal perfect central extension of $G/\Z(G)$.

 For some $\Gu$ and $D$ as in Notation \ref{Notation_Gu} and Section 3 of \cite{Maslowski} we apply Theorem \ref{hauptthm2}. By Remark \ref{rem_semisimple_char} the global part of \ref{hauptprop_maxext} and \ref{Star-cond-global} hold. Further by Theorem \ref{thm_Maslowski} there exists a bijection with the required properties. The characters in $\Irrl(\Bu)$ have, by Remark \ref{rem_Star_cond_local}, the property described in \ref{Star-cond-local}. Hence by Theorem \ref{hauptthm2} the inductive McKay condition holds for $G/\Z(G)$ and $\p$. 
\end{proof}

\bibliographystyle{alpha}
\bibliography{literatur09}
\end{document}